\documentclass[11pt]{article}
       \usepackage{amsfonts}
       \usepackage{stmaryrd}
       \usepackage{latexsym,amssymb,mathrsfs,fancyhdr}
       \font\tenmsb=msbm10
       \font\sevenmsb=msbm7
       \font\fivemsb=msbm5
       \catcode`\@=11
       \ifx\amstexloaded@\relax\catcode`\@=\active
       \endinput\else\let\amstexloaded@\relax\fi
       \def\spaces@{\space\space\space\space\space}
       \def\spaces@@{\spaces@\spaces@\spaces@\spaces@\spaces@}
       \def\space@.  {\futurelet\space@\relax}
       \space@.   %
       \def\Err@#1{\errhelp\defaulthelp@\errmessage{AmS-TeX error: #1}}
       \def\relaxnext@{\let\next\relax}
       \def\accentfam@{7}
       \def\noaccents@{\def\accentfam@{0}}
       \def\Cal{\relaxnext@\ifmmode\let\next\Cal@\else
       \def\next{\Err@{Use \string\Cal\space only in math mode}}\fi\next}
       \def\Cal@#1{{\Cal@@{#1}}}
       \def\Cal@@#1{\noaccents@\fam\tw@#1}
       \def\Bbb{\relaxnext@\ifmmode\let\next\Bbb@\else
       \def\next{\Err@{Use \string\Bbb\space only in math mode}}\fi\next}
       \def\Bbb@#1{{\Bbb@@{#1}}}
       \def\Bbb@@#1{\noaccents@\fam\msbfam#1}
       \newfam\msbfam
       \textfont\msbfam=\tenmsb
       \scriptfont\msbfam=\sevenmsb
       \scriptscriptfont\msbfam=\fivemsb

\def\Z{\mathbb{Z}}
\def\R{\mathbb{R}}
\def\T{\mathbb{T}}
\def\C{\mathbb{C}}

\@addtoreset{equation}{section}

\newcommand{\la}{\langle }
\newcommand{\ra}{\rangle }

\newcommand{\beq}{\begin{equation} }
\newcommand{\eeq}{\end{equation} }

\usepackage[german,french,english]{babel}
\usepackage{amsmath,amssymb}

\newtheorem{theo}{Theorem}[section]
\newtheorem{rem}{Remark}[section]

\newtheorem{prop}{Proposition}[section]
\newtheorem{lemma}{Lemma}[section]
\newtheorem{iteration lemma}{iteration Lemma}[section]

\begin{document}



\setlength{\columnsep}{5pt}
\title{General KAM theorems and their applications to invariant tori with prescribed frequencies  \footnote{The work was supported by the National Natural
Science Foundation of China(11371090) }}
 \author{Junxiang Xu \footnote{Corresponding author,  E-mail address: xujun@seu.edu.cn  } , Xuezhu Lu\footnote{ E-mail: Jennysundate@163.com} \\ Department of
Mathematics,
 Southeast University \\  Nanjing 210096,  China
  }

\maketitle
\begin{quote}
{\small  In this paper we develop some new KAM-technique to prove two general  KAM theorems for nearly integrable hamiltonian systems without assuming any non-degeneracy condition.
Many of KAM-type results (including the classical KAM theorem) are   special cases of our theorems under some non-degeneracy condition and some smoothness condition.
   Moreover,  we can  obtain some interesting results about KAM tori with  prescribed frequencies.
   }

  \end{quote}

\section{Introduction}\label{a}

 In this paper we   consider  the persistence of invariant tori of integrable hamiltonian system under small perturbation, which
    has been the fundamental problem of  hamiltonian system and also the  motivation of many KAM theorems \cite{A1,E1,K, Mel1, Mel2, Mos, Ms}.  As is well known,   KAM method becomes a mighty instrument to deal with  such that quasi-periodic problem with the notorious small divisors. The proof of KAM theorems are based on  the  KAM iteration, involved with certain small divisor condition  or non-degeneracy condition \cite{B, Bru, Ch,Pos1, P1, R1,R2}.
   In this paper, we  develop some new KAM technique to prove two general  KAM theorems without imposing  small divisor condition  or non-degeneracy condition, which   can be  applied to  diverse cases to obtain some interesting results.
These general  KAM theorems  make no sense if no   small divisor condition  or non-degeneracy condition is assumed.


Let  $ H(q,p)=h(p)+f(q,p)$,
where    $(q, p)\in \T^n\times D,$  with $\T^n$ the usual
$n$-dimensional torus and $D$ a bounded simply connected open domain
of $\R^n.$ $h(p)$ and $f(q,p)$ are real analytic on $\bar D$ and
$\bar D \times \T^n$, respectively. The corresponding hamiltonian
system is
\begin{align}
\dot q &=\;\;\, H_p(q,p)=h_p(p)
+ f_p(q,p)\nonumber \\ \dot p & =-H_q(q,p) = \;\;\;\;\;\;\;\;\; -f_q(q,p)
 \label{s1}
\end{align}

 At first, by a transformation $p=\xi+I$ and $q=\theta$, $\xi\in D,$ we introduce a parameter $\xi$ and then the hamiltonian system (\ref{s1}) is equivalent to
  a parameterized hamiltonian system:
  \begin{align}H(q,p)& =h(\xi)+\langle
h_p(\xi),I\rangle + f_h(\xi;I) + f(\theta, \xi+I)\nonumber \\
& =e+\langle\omega(\xi), I\rangle + P(\xi;\theta,I),\nonumber
\end{align} where
$$ e=h(\xi), \ \omega(\xi)=h_p(\xi), \  f_h(\xi;I)=h(I+\xi)-h(\xi)-\langle
h_p(\xi),I\rangle ,$$ $$P(\xi;\theta,
I)=f_h(\xi;I) + f(\theta, \xi+I),$$ and $\xi\in\Pi\subset  D$ is regarded
as parameter.
 Here $e$ is an
energy constant  and   usually is omitted in KAM steps. $\omega: \xi \to
\omega(\xi)$ is called frequency mapping, and
 $P$ is  a small perturbation term.

This technique of  introducing  parameter was first used  by P\"oschel in \cite{P1}, leading to separation of  invariant tori and their frequencies in KAM iteration.

Then the corresponding hamiltonian
system becomes
 \begin{equation}
 \left\{
 \begin {aligned}
\dot \theta &=\;\;\, H_I=\omega(\xi)  + P_{I}(\xi;\theta,I) \\
\dot I &=-H_\theta=\;\;\;\;\;\;\;\;  -P_{
\theta}(\xi;\theta,I)
\end{aligned}
\right. \label{s2}
\end{equation}
  Thus, the  persistence of a family of invariant tori $\T^n\times\{p\}$ for   (\ref{s1}) is reduced to that of invariant tori
  $\T^n\times\{0\}$ with  frequencies $\omega(\xi)$
  for  hamiltonian  system (\ref{s2}) with the parameter $\xi\in \Pi$.


Without loss of generality, we consider the  parameterized
hamiltonian system (\ref{s2}) with  $H=H(\xi;\theta, I) = \langle\omega(\xi), I \rangle + P(\xi;\theta, I),$ where $P$ is a perturbation term.

Let  $0<\alpha<1, \tau>n-1$ and
\begin{equation}\label{dio}
O_{\alpha,\tau}=\bigl\{\omega \in \R^n:  \
\bigl|\langle\omega, k \rangle \bigr|\ge
\frac{\alpha}{|k|^\tau}, \,  \forall \,  k\in \Z^n \setminus \{0\}\bigr\}.
\end{equation}

  If $P=0$,   for every parameter $\xi\in\Pi$ the system $(\ref{s2})$ admits an  invariant torus  $ T^n\times \{0\}$ with   frequency $\omega(\xi)$. The classical KAM theorem says that  if the frequency mapping is
non-degenerate in Kolmogorov's  sense:
\begin{equation} \mbox{det}( \partial_{\xi} \omega  )=\mbox{det}(h_{pp})\ne 0,
\nonumber\end{equation}
  then for most  $\xi\in \Pi$ such that $\omega(\xi)\in O_{\alpha,\tau}$, the  invariant tori with    frequencies $\omega(\xi)$  will survive of arbitrarily sufficiently small perturbations
  \cite{ A1,E1, K, Ms,P1,Pos1}.

 Later,  Kolmogorov's  non-degeneracy condition has  been weakened to    R\"ussmann's non-degeneracy condition \cite{R1,R2,X1, Se2}:
\begin{equation}
a_1\omega_1(\xi)+a_2\omega_2(\xi)+\cdots +a_n\omega_n(\xi)\not\equiv
0\,\,\, \mbox{on}\ \Pi ,\label{rn}
\end{equation} for all $
(a_1,   \ldots, a_n) \in \R^n\setminus\{0\} $.
That means, under the condition (\ref{rn}), for most  $\xi$ in the sense of Lebesgue  measure,    the perturbed system
$(\ref{s2})$ still has  invariant tori with  frequencies in $O_{\alpha, \tau}$.
However, since  the range of the frequency mapping $\omega$ may be on a sub-manifold,
  the frequencies of persisting invariant tori  may not come from  unperturbed ones.
Thus it is difficult to provide  accurate information about the frequencies of KAM tori.



More recently, some authors  turn to study the  persistence of invariant tori with prescribed frequency.
In the paper \cite{Xu4}, assuming $\omega_0\in O_{\alpha,\tau}$ and   $\mbox{deg}(\omega, \Pi, \omega_0)\ne 0$, the authors proved
the perturbed parameterized system (\ref{s2}) still has an invariant torus with
 $\omega_0$ as its frequency, i.e., the torus with the prescribed frequency  $\omega_0$ persists under small perturbations.

 However, the  result in \cite{Xu4}
  cannot be generalized to lower dimensional  elliptic  invariant tori.
In the  Kolmogorov  non-degenerate case, Bourgain considered the following  hamiltonian
 $$H(\omega;\theta, I, z, \bar z)=
\langle\omega, I\rangle +  \Omega(\omega)z\,\bar z+P(\omega;\theta, I, z,\bar z),$$
  and  obtained a similar result for lower dimensional  elliptic  invariant tori
 \cite{Bou2}.  More precisely,
 suppose $\omega_0\in O_\alpha$ and  $(\omega_0, \Omega_0)=(\omega_0, \Omega(\omega_0))$ satisfy the first Melnikov  condition, then  for most of sufficiently small $\lambda$, there exists $\xi$ such that the above perturbed hamiltonian
  has an elliptic lower dimensional invariant torus $\T^n\times\{0,0,0\}$ with the frequency $(1+\lambda)\omega_0.$


In this paper, we  are mainly interested in the persistence of invariant tori  with prescribed frequency.
  For this purpose, we will develop a new technique of  KAM iteration
    to separate non-degeneracy condition from KAM iteration. The key  lies in   an explicit  extension of small divisors to the  parameter definition domain.
       Our extension of small divisors   always  works even though the  small divisor condition does not hold for every $\xi\in\Pi$.   Thus the constructed symplectic transformation  and  the new perturbation are well defined for all parameters. However, only for these parameters such that the small divisor
    condition holds, the new hamiltonian is exactly from the original one  under the  symplectic transformation; otherwise,  we only  obtain a  formal new hamiltonian, it may have no relation with the previous hamiltonian  and thus cannot provide any useful information.

To be more precise, let $\alpha>0, \tau>n-1$ and a  family of parameterized hamiltonian be
$$\{H(\xi; \theta,I)=\langle\omega(\xi),I\rangle +P(\xi; \theta,I): \ \xi\in\Pi\}.$$
By our KAM iteration, we can have a  family of parameterized normal hamiltonian  $$\{H_*(\xi; \theta,I)=\langle\omega_*(\xi),I\rangle +P_*(\xi; \theta, I): \ \xi\in\Pi\},$$
   where the frequency mapping $\omega_*(\xi)$ is a small perturbation of $\omega$ and $P_*=O(I^2).$
For $\xi\in \Pi,$ if $\omega_*(\xi)\in O_\alpha$, the original  hamiltonian   $H(\xi;\cdot)$ is just normalized to   $H_*(\xi;\cdot)$ and then has an invariant torus with frequency  $\omega_*(\xi)$.
 If  $\omega_*(\xi)\notin O_\alpha,$  we cannot have  any   relation between  $H(\xi;\cdot)$ and  $H_*(\xi;\cdot)$; in this case  $H_*(\xi;\cdot)$ does not provided  any information of $H(\xi;\cdot).$
 Thus, if   $\omega_*(\xi)\notin O_\alpha$ for all $\xi\in\Pi$, our result makes no sense.

\section{ Main Results \label{b}}

To state our theorems, we first give some notations.
Define a small neighborhood of $\T^n\times\{0\}$ by
$$D(s,r)=\bigl\{(\theta,I)\in \C^n\times \C^n: \,  |\mbox{Im}\ \theta |_\infty \le s,
\, |I|_1 \le r\bigr. \bigr\},$$ where $|\mbox{Im}\ \theta |_\infty
=\max_{1\le i\le n}|\mbox{Im} \theta_i|,$ $|I|_1=\sum_{1\le i\le
n}|I_i|$.
Let $\Pi\subset \R^n$ be a bounded connected closed domain.

Consider a parameterized  hamiltonian
\begin{equation}\label{LH1}
H(\xi;\theta, I)=
\langle\omega(\xi), I\rangle + P(\xi;\theta, I).
\end{equation}
Suppose $H(\xi;\theta, I)$ is real analytic in
$(\theta, I)\in D(s,r)$ and $C^m$-smooth in $\xi\in \Pi$ with  $m\ge 0$.

We expand $P(\xi;\theta,I)$ as the Fourier series with respect to
$\theta$
$$P(\xi;\theta,I)=\sum_{k\in \Z^n  }P_{k }(\xi;I)\, e^{{\rm i}\langle  k,\theta
\rangle   }.$$
Let $\Z^n_+$  consist  of all the integer vectors with non-negative components, and then   $P_{k}(\xi;I)=\sum_{\ell\in \Z_+^n }P_{k,\ell}(\xi)\, I^\ell $.

Define
$$\|P\|_{\alpha,\Pi\times D(s,\,r)}= \sum_{k }
\|P_{k}\|_{\Pi;r} \, e^{s|k|} ,$$
 where
  $\|P_{k}\|_{\Pi;r}=\sup_{ |I|_1\le r } |\sum_{\ell\ge 0 }\|P_{k,\ell}\|_{\alpha, C^m(\Pi)}I^\ell | $
with the weighted norm
$$\|P_{k,\ell}\|_{\alpha, C^m(\Pi)}=\max\limits_{|\beta|\le m}\alpha^{|\beta|}\max\limits_{\xi\in\Pi}\bigl|\frac{\partial^\beta P_{k,\ell}(\xi)}{\partial^\beta\xi}\bigr|.$$
 The weight  $\alpha$ is supplemented so that the  relevant KAM estimates in the sequel can be written in a succinct way.

\vskip 0.2 true cm \begin{theo}\label{TT2}
Consider the hamiltonian   (\ref{LH1}) and suppose $H$
 is real analytic in $(\theta, I)$
on $ D(s,r)$ and $C^m$-smooth in  $\xi$ on $\Pi$. Let $O_{\alpha,\tau}$  be defined  as (\ref{dio}).
For any $0<\alpha\le 1,$ $\tau>n-1$ and $m\ge 0,$   there exists a sufficiently small  $\gamma>0,$
 such that
   if
   $$\|P\|_{\alpha, \Pi\times D(s,\,r)}=\epsilon\le \alpha r s^{\tau'} \gamma \ \mbox{with} \, \tau'=n+(m+1)\tau+m, $$
 then  there exist a  family of symplectic mappings $ \{\Phi_*(\xi;\cdot )\, | \ \xi\in\Pi\}$  and  a  family of hamiltonian  $ \{H_*(\xi;\cdot )\, | \ \xi\in\Pi\}$ such that   the following conclusions  hold:


  \vskip 0.2 true cm
  \noindent
  (i)
 $\Phi_*(\xi;\theta, I)$  is  analytic in $(\theta,I)$ on $ D(s/2,\,r/2)$ and $C^m$-smooth in $\xi$ on $\Pi$,
    and maps $D(s/2,\,r/2)$ into $D(s,\,r).$
      Moreover, $$\|W(\Phi_* -\mbox{id}) \ \|_{\alpha, \Pi\times D(s/2,\,r/2)}\le c\gamma,$$ where $W=\mbox{diag}(\rho^{-1}I_n, r^{-1}I_n)$ with $\rho=s/20$ and   $I_n$  the $n$-th  unit matrix.

   \vskip 0.2 true cm
  \noindent
  (ii)
   $$H_*(\xi;\theta, I)=\langle\omega_*(\xi), I\rangle +P_*(\xi;\theta, I) $$
is  analytic in $(\theta,I)$ on $ D(s/2,\,r/2)$ and  $C^m$-smooth in $\xi$ on $\Pi,$ with the estimates
    $$\|\omega_*-\omega\|_{\alpha, C^m(\Pi)}\le 2\epsilon/r, \ \ P_*(\xi;\theta, I)=O(I^2).$$

   \vskip 0.2 true cm
  \noindent
  (iii)  If  $\omega_*(\xi)\in O_{\alpha,\tau},$    we have
 $$H\circ\Phi_*(\xi;\theta,I)=H_*(\xi;\theta, I).$$
 Thus the hamiltonian   $H(\xi;\cdot) $ has an invariant torus  $\Phi_*(\xi;\T^n,0)$ with the frequency $\omega_*(\xi).$

\end{theo}


Next we consider the perturbation of  elliptic  lower dimensional invariant tori  and establish an analogous KAM theorem.
When it causes no confusion, we still employ the same notations to denote the variables and sequencies.

\vskip 4 pt
Define a complex neighborhood of $\T^n\times\{0,0,0\}$ by
$$D(s,r)=\{(\theta, I, z, \bar z)\in \C^n\times \C^n\times \C^{\bar n}\times \C^{\bar n}: \ |\mbox{Im}\theta |_{\infty}\le s, |r|_{1}\le r^2,  |z|_2\le r, |\bar z|_2\le r\}.$$

Consider a parameterized  hamiltonian
\begin{equation}\label{WH1}
  H(\xi;\theta, I, z, \bar z)=
\langle\omega(\xi), I\rangle +  \langle\Omega(\xi), z\, \bar z\rangle+P(\xi;\theta, I, z,\bar z)
\end{equation}
defined for $(\xi;\theta, I, z, \bar z)\in \Pi\times D(s,r)$, where
\begin{equation}\label{WH2}
\Omega=(\Omega_1, \ldots, \Omega_{\bar n}) \quad  \mbox{and}~~ z\bar z=(z_1\bar z_1,   \ldots, z_{\bar n}\bar z_{\bar n}).
\end{equation}
The associated  symplectic structure is
$$\sum_{i=1}^{n}d I_i \wedge d\theta_i +{\rm i}\sum_{j=1}^{\bar n}dz_j\wedge d\bar z_j,$$ with ${\rm i}=\sqrt{-1}.$
Suppose $H$ is real analytic in
$(\theta, I, z, \bar z)$ on  $ D(s,r)$ and $C^m$-smooth in $\xi$ on $\Pi$ with  $m\ge 0$.

Expand $P$ as the Fourier series with respect to
$\theta$
$$P(\xi;\theta,I,z,\bar z)=\sum_{k\in \Z^n  }P_{k}(\xi; I, z, \bar z)\, e^{{\rm i}\langle  k,\theta
\rangle   } .$$
Let
 $$P_{k}(\xi;I, z,\bar z)=\sum_{\ell_1\in \Z_+^n, \ell_2, \ell_3\in Z_+^{\bar n}}P_{k,\ell}(\xi)\, I^{\ell_1}z^{\ell_2}{\bar z}^{\ell_3},\ \ \ell=(\ell_1, \ell_2, \ell_3)  $$
  Then we define
$$\|P\|_{\alpha, \Pi\times D(s,\,r)}= \sum_{k}
\|P_{k}\|_{\Pi;r} \, e^{s|k|} ,$$
where
$$\|P_{k}\|_{\Pi;r}=\sup_{|I|_{1}\le r^2, |z|_2\le r, |\bar z|_2\le r } \bigl|\sum_{\ell }\|P_{k,\ell}\|_{\alpha, C^m(\Pi)}I^{\ell_1}z^{\ell_2}{\bar z}^{\ell_3} \bigr|.$$

 Set $\mathcal{Z}=\{(k,l)\in  \Z^n\times \Z^{\bar n}:~k\neq 0,  |l|\le 2\} $ and $\mathcal{L}=\{l\in \Z^{\bar{n}}:~1\leq|l|\leq 2\}$.
  For fixed constants $0<\alpha<1 $ and $ \tau>n-1$, define $\tilde O_{\alpha,\tau}\subset \R^{n}\times \R^{\bar n}$ as
  \begin{equation}\label{dio*}
  \tilde O_{\alpha,\tau}=\bigl\{(\omega, \Omega):~  \bigl|\langle\omega, k \rangle +\la l,\Omega\ra \bigr|\ge
\frac{\alpha}{|k|^{\tau}},~   (k,l)\in \mathcal{Z}~\text{and}~|\la l,\Omega\ra |\ge \alpha ,~l\in \mathcal{L} \bigr\}.
\end{equation}

\vskip 0.2 true cm \begin{theo}\label{TT3}
Consider  the hamiltonian  $(\ref{WH1})$ and suppose $H$
  is real analytic in $(\theta, I, z,\bar z)$
on $ D(s,r)$ and $C^m$-smooth in  $\xi$ on $\Pi$.
 Let $\tilde O_{\alpha,\tau}$  be defined  by (\ref{dio*}).  For any $1\ge \alpha>0$, $\tau >n-1$ and $m\ge 0,$  there exists a sufficiently small  $\gamma>0,$
 such that
   if
   $$\|P\|_{\alpha, \Pi\times D(s,\,r)}=\epsilon \le \alpha r^2 s^{\tau'} \gamma  \, \mbox{with} \, \tau'=n+(m+1)\tau+m, $$
  then  there exists a   family of symplectic mappings $ \{\Phi_*(\xi;\cdot )\ |  \, \xi\in\Pi\}$  and a   family of hamiltonian  $ \{H_*(\xi;\cdot )\, | \ \xi\in\Pi\}$ such that    the following conclusions  hold:


    \vskip 0.2 true cm
  \noindent
  (i)
 $\Phi_* (\xi;\theta,I, z,\bar z)$  is  analytic in $ (\theta,I, z,\bar z)$ on $ D(s/2,\,r/2)$ and $C^m$-smooth in $\xi$ on $\Pi$,
    and maps $D(s/2,\,r/2)$ into $ D(s,\,r)$.
    Moreover,
     $$\|W (\Phi_* -\mbox{id} ) \ \|_{\alpha, \Pi\times D(s/2,\,r/2)}\le c\gamma,$$ where
   $W=\mbox{diag}(\rho^{-1}I_n, r^{-2}I_n, r^{-1}I_{\bar{n}}, r^{-1}I_{\bar{n}})$ with $\rho=s/20$ and
   $I_n$ being the $n$-th unit matrix.

   \vskip 0.2 true cm
  \noindent
  (ii)
   $$H_*(\xi;\theta,I, z, \bar z)=\langle\omega_*(\xi), I\rangle +\langle\Omega_*(\xi), z\bar z\rangle+ P_*(\xi;\theta, I, z, \bar z) $$
  is   analytic in $(\theta,I, z,\bar z)$ on $ D(s/2,\,r/2)$ and $C^m$-smooth in $\xi$ on $\Pi$ , with the estimates
    $$\|\omega_*-\omega\|_{\alpha, C^m(\Pi)}\le 2\epsilon/r^2, \quad \|\Omega_*-\Omega\|_{\alpha, C^m(\Pi)}\le 2\epsilon/r^2,$$ and
     $$\ \partial^{\ell_1}_I\partial^{\ell_2}_z\partial^{\ell_3}_{\bar z}P_*(\xi;\theta, 0,0,0)=0, \ \forall \,  2|\ell_1|+|\ell_2|+|\ell_3|\le 2.$$

   \vskip 0.2 true cm
  \noindent
  (iii)
  If $(\omega_*(\xi), \Omega_*(\xi))\in \tilde O_{\alpha,\tau}$, we have
 $$H\circ\Phi_*(\xi;\theta,I, z, \bar z)=H_*(\xi;\theta,I, z, \bar z).$$
Thus the hamiltonian   $H(\xi;\cdots) $  has an elliptic lower  invariant torus $\Phi_*(\xi;\T^n,0, 0, 0 )$     with $\omega_*(\xi)$ being the tangential  frequency and $\Omega_*(\xi))$  the normal frequency.

\end{theo}

\begin{rem}\label{EX}
 Theorems \ref{TT2} and \ref{TT3} imply that  the existence of KAM tori is equivalent to  whether   the final frequencies $\omega_*$ belong to $O_{\tau,\alpha}$ or    $(\omega_*,\Omega_*)$ belong to $\tilde O_{\tau,\alpha}.$ Note that in our theorems  we do not need  any non-degeneracy assumption  and  any strict smoothness condition for parameter as in the previous KAM theorems; thus our results are more general.
\end{rem}

\section{Stability of Diophantine Frequency  \label{c}}


For application of  Theorems \ref{TT2} and \ref{TT3}, in this section we make some preliminaries to  explore  existence of  Diophantine frequencies in  the image set $\{\omega_*(\xi):  \xi\in \Pi\}$ or  $\{(\omega_*(\xi),\Omega_*(\xi)):\xi\in\Pi\}$ .
 Note that  $\omega_*=\omega+\hat\omega$ is only  a small perturbation of $\omega$ and
        $\|\cdot \|_{C^m(\Pi)}\le \alpha^{-m} \|\cdot \|_{\alpha, C^m(\Pi)}.$ By the theorems, we have  $\|\hat\omega\|_{C^m(\Pi)}\le \sigma,$ where \ $ \sigma=\frac{2\epsilon}{ r \alpha^m}$  or $ \sigma=\frac{2\epsilon}{r^2\alpha^m}$ in the case of elliptic lower dimensional tori. This observation illustrates that  the stability of Diophantine frequency is quite important for our problem.

\vskip 0.3 true cm
{\it Stability of prescribed frequency.}

Let $\omega_0\in \omega(\Pi)=\{\omega(\xi):  \xi\in \Pi\}$ and  $\omega_*=\omega+\hat\omega$.
If there exists  a sufficiently small constant $\sigma_0>0$, such that if $\|\hat\omega\|_{C^m(\Pi)} \le \sigma_0,$
we have a  $\lambda$ with $|\lambda| << 1$,  such that $(1+\lambda) \omega_0\in \omega_*(\Pi),$ we say  the direction of $\omega_0$ is stable  in $\omega(\Pi);$
in particular, if $\lambda=0$,    we say $\omega_0$ is stable.



\vskip 0.2 true cm
The above definition suggests that the stability of Diophantine frequency corresponds to the  persistence of invariant tori with the prescribed frequency.
When the direction of  a diophantine frequency is stable, there  exists an invariant torus with the  frequency only being a dilation of the prescribed frequency. This kind of invariant tori  carry certain  information of frequencies from the integrable system.




\vskip 4 pt
In what follows,  $\omega\in \R^n$ always indicates a row vector.   The  notation $\|\cdot \|_{m}$ is used in place of  $\|\cdot \|_{C^m(\Pi)}$ for short, especially, $\|\cdot \|=\|\cdot \|_0.$ We always denote by $\omega_*=\omega+\hat\omega$ and by $\Omega_*=\Omega+\hat\Omega$     small perturbations of $\omega$ and $\Omega$, respectively.

\begin{lemma}\label{l4}
 Let  $\Pi \subset \R^n$, $\omega(\xi)    \in \R^n$ and  $  \Omega(\xi)\in \R$ belong to
 $ C^1(\Pi)$ with  $$\mathrm{rank}\ ( \partial_{\xi}\omega ) =n \ \ \mbox{and} \ \  | \Omega(\xi)|\ge c>0,\ \forall \xi\in\Pi.$$
 Denote by $U=(\omega, \Omega)^T$ the transpose of $(\omega, \Omega)$.   Let $A(\xi)= ( \partial_{\xi} U,U) $ and $$ \partial_{\xi} U =(\partial_{\xi_1} U,  \ldots, \partial_{\xi_n} U ).$$
 Suppose $\mathrm{rank} (A(\xi))  =n+1,$ for all $ \xi\in\Pi.$
   Let $\tilde\omega=\omega/\Omega$. Then we have  $$\mathrm{rank}\ ( \partial_{\xi} \tilde\omega )=n,\quad \forall \xi\in\Pi.$$
\end{lemma}

\begin{proof}
Set $V=(\tilde\omega, 1)^T$ and then
$U=\Omega \cdot V.$
 It follows that
$$A=( \partial_{\xi_1}\Omega\cdot V+\Omega \cdot \partial_{\xi_1}V, \ \ldots, \partial_{\xi_n}\Omega\cdot V+\Omega \cdot \partial_{\xi_n}V, \ \Omega\cdot V).$$

  Let
$$ B=( \partial_{\xi_1}V,   \ldots,   \partial_{\xi_n}V,~  V).$$
Note that $\partial_{\xi_j}\Omega$ are scalar functions  for all arbitrary $j $. Therefore,
 $$\mathrm{rank}(B(\xi))=\mathrm{rank}(A(\xi))=n+1$$ and then $\mathrm{rank}\ ( \partial_{\xi} \tilde\omega )=n.$
\hfill $\Box$
\end{proof}


\begin{lemma}\label{l5}
 Let   $\omega(\xi)    \in \R^n$  belong to $ C^1(\Pi)$  and $\omega^T$ be its transpose.   Suppose that for all $ \xi\in\Pi$ we have
 $$\mathrm{rank}\ ( \partial_{\xi} \omega  )=n-1  \ \  \text{and}~~ \mathrm{rank}\ (\partial_{\xi} \omega^T, \omega^T )=n.$$
 Denote by    $$\tilde\omega=(\omega_1,\ldots,\omega_{j-1},\omega_{j+1},\ldots, \omega_n ),\ \ \tilde\xi=(\xi_1, \ldots,  \xi_{i-1}, \xi_{i+1}, \ldots, \xi_n). $$
 Then, for any  $\xi\in\Pi$ there exist   $ i$ and $ j $ with $1\leq i, j\leq n$  such that $\mathrm{rank}\ ( \partial_{\tilde\xi} \tilde\omega  )=n-1.$ Moreover,  if $\omega_j(\xi)\ne 0$,   we have
  $\mathrm{rank}\ ( \partial_{\tilde\xi}   \frac{\tilde\omega}{ \ \omega_j}  )=n-1.$

  \end{lemma}

\begin{proof}
This lemma can be proved by  directly applying   Lemma \ref{l4}.
\hfill $\Box$
\end{proof}

\begin{lemma}\label{l2}
 Let     $\omega(\xi)=(\omega_1(\xi),\ldots, \omega_n(\xi))$ and $\omega_*=\omega+\hat\omega  $  belong to $ C^1(\Pi),$
   where $$ \xi=(\tilde\xi, \xi_n)\in\Pi= \widetilde{\Pi}\times [ \xi_{0n}-\beta, \xi_{0n}+\beta]\subset \R^{n-1 }\times \R. $$
Set $\omega_0=\omega(\xi_0) $,  $ \xi_0=(\tilde\xi_0, \xi_{0n})\in \Pi$,
$$\omega^{\flat}=(\omega_1, \ldots, \omega_{n-1}) \quad \text{and}~~ \hat{\omega}^{\flat}=(\hat{\omega}_1, \ldots, \hat{\omega}_{n-1}).$$
Let  $  \lambda(\xi)=\omega_{0n}^{-1}\cdot \omega_{*n}(\xi) -1.$
 Suppose for $\xi\in \Pi$, $|\omega_n(\xi)|\ge c >0, $ and the Jacobian matrix $ \partial_{\tilde\xi}  (\frac{\omega^{\flat}}{\omega_n}) $ is non-degenerate.
 Then,
there exist  sufficiently small positive constants $\sigma_0 $  and $\delta_0$, such that  if $\|\hat\omega\|=\sigma<\sigma_0$ and $|\xi_n-\xi_{0n}|\le \delta_0$,
 there exists a unique
 $\xi_*=\xi_*(\xi_n)\in\Pi,$  which is continuously differentiable in $\xi_n,$
  such  that $$\omega_*(\xi_*)=(1+\lambda(\xi_*)) \omega_0.$$
Moreover,   $\lambda(\xi_*)=O(|\xi_n-\xi_{0n}|+\sigma).$

\end{lemma}

\begin{proof}
Rewrite as $\omega_*=(\omega_{n}+\hat\omega_n)\cdot \tilde{\omega}_*,$
where  $\tilde{\omega}_*=(\omega_n^{-1}\cdot \omega^{\flat} +a(\xi), \ 1) $ and
$$a(\xi)=-\frac{\hat\omega_n  }{\omega_n(\omega_n+\hat\omega_n)}\cdot \omega^{\flat} +\frac{1}{\omega_n+\hat\omega_n}\cdot \hat\omega^{\flat}.$$
It is easy to verify  $\|a\|\le c_1\sigma.$

Note that the above  functions are all uniformly  continuous in $(\tilde\xi, \xi_n)$. The assumption also implies   $\omega_n^{-1}\cdot \omega^{\flat}$ is non-degenerate uniformly with respect to $\tilde\xi$. Hence, there exists a small $\delta_0$ such that  if $\sigma$ is sufficiently small, for   $|\xi_n-\xi_{0n}|\le \delta_0$,  we have a unique  $\tilde\xi_*=\tilde\xi_*(\xi_n)$ and $\xi_*=(\tilde\xi_*(\xi_n),\xi_n)$, such that
$$\omega_n^{-1} (\xi_*)\cdot \omega^{\flat}   (\xi_*)+a(\xi_*)=\omega_{0n}^{-1} \cdot \omega_{0}^{\flat} .$$
 Moreover,  $\tilde\xi_*$  is differentiable in $\xi_n,$  and satisfies
 $$|\tilde\xi_*(\xi_{n})-\tilde\xi_0|\le c_2|\xi_n-\xi_{0n}|+c_3\sigma.$$
In view of
$\omega_{0n}\cdot \tilde{\omega}_* (\xi_*)=\omega_{0},$ it is easy to see that
  $$\omega_*(\xi_*)=(1+\lambda(\xi_*))\omega_0.$$
Moreover,  $$\lambda(\xi_*)=\omega_{0n}^{-1}\cdot\bigl(\omega_{n}(\xi_*)-\omega_{n}(\xi_0)+\hat\omega_n(\xi_*) \bigr)=O(|\xi_n-\xi_{0n}|+\sigma)$$ as $\xi_n\to \xi_{0n}$ and $\sigma\to 0.$
\hfill$\Box$
\end{proof}

\begin{prop}\label{l6}
 Let     $\omega(\xi)\in \R^n$ and $\omega_*(\xi)=\omega(\xi)+\hat\omega(\xi)  $  belong to $ C^1(\Pi)$.
Suppose the following Bruno non-degeneracy condition hold:
\begin{equation}\label{BR}
\mathrm{rank}\ ( \partial_{\xi} \omega  )=n-1  \ \  \text{and}~~ \mathrm{rank}\ (\partial_{\xi} \omega^T, \omega^T )=n,\ \ \forall \xi\in\Pi.
\end{equation}
 Then
  $\omega(\Pi)\cap O_{\alpha,\tau} \ne \varnothing.$  Moreover, let   $$\omega_0=\omega(\xi_0)\in O_{ \alpha, \tau},\quad \xi_0\in \Pi.$$
There exists  sufficiently small positive constants $\delta_0 $ and $\sigma_0 $   such that if $\|\hat \omega\|=\sigma<\sigma_0, $  the set $\omega_*(\Pi)$ contains a continuously differentiable one-parameter family of   Diophantine frequencies with the form  $(1+\lambda(\eta))\omega_0$, where   $\lambda$ is continuously differentiable for $|\eta |\le \delta_0 $, and satisfies    $\lambda(\eta)=O(|\eta|+\sigma).$
  \end{prop}

\begin{proof}
 It is well known from the Bruno non-degeneracy condition that if $\alpha$ is sufficiently small, $\omega(\Pi)\cap O_{ \alpha,\tau}$ is nonempty.

Since $\omega_0 $  is Diophantine, it follows  $\omega_{0j}\ne 0$ for $1\le j\le n.$
Applying  Lemma \ref{l5}, there exists a small neighbor $\Pi_0$  of $\xi_0$ in $\Pi$  and  $ i, j $,
 such that for all $\xi\in\Pi_0$,   we have
$\omega_j(\xi)\ne 0$,
$$\mathrm{rank}\ ( \partial_{\tilde\xi }  \tilde\omega )=n-1 \quad\text{and}~~ \mathrm{rank}\ ( \partial_{\tilde\xi}  \frac{\tilde\omega}{\ \omega_j} )=n-1, $$
 where $\tilde\omega$ and $\tilde\xi$ are defined  as in Lemma \ref{l5}.

Then  Lemma \ref{l2} ensures
   sufficiently small constants  $\sigma_0>0$  and $\delta_0>0$, such that     for  $\|\hat\omega\|=\sigma<\sigma_0$ and  $|\xi_j-\xi_{0j}|< \delta_0,$ we have a unique $\xi_*=\xi_*(\xi_j)$, that  is continuously  differentiable   in  $\xi_j,$
 such that  $$\omega_*(\xi_*)=(1+\lambda(\xi_*))\omega_0.$$
 Moreover,   $\lambda(\xi_{*})=O(|\xi_j-\xi_{0j}|+\sigma)).$  Let $\xi_j=\eta+\xi_{0j}$,  then we finish the proof.
 \hfill $\Box$

  \end{proof}

\begin{lemma}\label{l7}
 Let $ \omega_0=(\omega_{01}, \omega_{02})$ satisfy the Diophantine condition
\begin{equation}\label{TJ}
|k_1\omega_{01}+ k_2\omega_{02}|\ge \frac{\alpha}{|k|^{\tau}}, \ \forall k\in \Z^2\setminus\{0\},
\end{equation}
where  $0<\alpha<1$ and $\tau>1$.
Set $f_k(\lambda)= k_1(\omega_{01}+\lambda)+k_2\omega_{02},$ and  $$\Pi_{\lambda_0}=\biggl\{\lambda\in [0,\lambda_0]:~ |f_k(\lambda)|\ge  \frac{\alpha}{2|k|^{2\tau+2}}, ~\forall k\in \Z^2\setminus\{0\}\biggr\}.$$
 Then  $\Pi_{\lambda_0}$ is a non-empty subset with  $\mathrm{meas}(  [0, \lambda_0]\setminus\Pi_{\lambda_0})=o(\lambda_0)$ as $\lambda_0\to 0$.
   \end{lemma}

\begin{proof}
Note that $\omega_{01},\omega_{02}\ne 0$. Without loss of generality, assume $|\lambda_0 |\le \frac{1}{2}|\omega_{01}| .$
 Observe that there exist  positive constants  $c_1$, $c_2$ and $ c_3$ such that
 if $ |k_1|\ge c_1|k_2|$ or $ |k_2|\ge c_2|k_1|$,  $|f_k(\lambda)|\ge c_3>0$. Then
 $ | f_k(\lambda) |\ge  \frac{\alpha}{2|k|^{2\tau+2}}$   holds  for sufficiently small $\alpha $.
 Hence, we consider the case of   $|k_2|/c_2 < |k_1|< c_1|k_2|.$

 If $|\lambda|\le \frac{\alpha}{2|k|^{\tau+1}},$  the  Diophantine assumption (\ref{TJ}) implies  $|f_k(\lambda)|\ge \frac{\alpha}{2|k|^{\tau}}.$
 Consequently, we consider these $k$ satisfying $  \frac{\alpha}{2|k|^{\tau+1}}<  \lambda_0 $ and  $|k_2|/c_2 < |k_1|< c_1|k_2|,$
 and denote by $ N_{\lambda_0}$ the set consisting of these $k$.

 For $k\in N_{\lambda_0}$, define the resonant set by
  $$\Delta_k=\biggl\{\lambda\in [ \frac{\alpha}{2|k|^{\tau+1}}, \lambda_0]:~  | f_k(\lambda) | <  \frac{\alpha}{2|k|^{2\tau+2}}\biggr\}.$$
 Then we have $\mathrm{meas}(\Delta_k)\le  \frac{\alpha}{2|k_1|\cdot |k|^{2\tau+2}}.$
In view of  $ [0, \lambda_0]\setminus\Pi_{\lambda_0}\subset \bigcup_{k\in N_{\lambda_0}}\Delta_k,$
 thus  $$\mbox{meas}( [0,\lambda_0]\setminus\Pi_{\lambda_0})\le \sum_{k\in N_{\lambda_0}}\frac{1}{|k_1|}\cdot \frac{\alpha}{2|k|^{2\tau+2}}.$$
It easily   follows
\begin{equation*}
\mbox{meas}( [0,\lambda_0]\setminus\Pi_{\lambda_0})\le c\lambda_0 \sum_{k_2>c(\frac{\alpha} {2\lambda_0})^{\frac{1}{\tau+1}} }\frac{1}{|k_2|^{\tau+1}}\le c\lambda_0 \bigl(\frac{\lambda_0}{\alpha}\bigr)^{\frac{\tau-1}{\tau+1}}.\tag*{ $\Box$ }
\end{equation*}
 \end{proof}

\begin{prop}\label{l7*}
 Let  $ \omega_0=(\omega_{01},  \omega_{02})$ satisfy   $(\ref{TJ})$ and $\omega_*=\omega_0+\hat\omega(\epsilon) $,  where
   $\hat\omega(\epsilon)$ is continuous  in the small parameter  $\epsilon\in [0,\epsilon_0]$ with $\hat\omega(0)=0.$  Then there exists a non-empty set
 $I^*_{\epsilon_0}\subset [0,\epsilon_0]$ with continuous carnality such that for $\epsilon\in I^*_{\epsilon_0}$, we have
  $$|\langle k, \omega_*(\epsilon) \rangle |\ge \frac{\alpha}{2|k|^{2\tau+2}}, \ \forall k\in \Z^2\setminus\{0\}.$$
  Moreover, when
  $\hat\omega(\epsilon)$ is continuously differentiable  on $[0, \epsilon_0]$ with $\|\hat\omega (\epsilon)\|_1\le c$,
 $I^*_{\epsilon_0}$ has positive measure.
    \end{prop}

\begin{proof}
Rewrite as $\omega_*=\frac{\omega_{*2}}{\omega_{02}}( \omega_{01}+\lambda(\epsilon), \omega_{02}),$  where
$\lambda=\frac{\omega_{02}}{\omega_{*2}}\cdot \hat\omega_{1}-\frac{\omega_{01} }{\omega_{*2}}\cdot \hat\omega_{2}.$
Let $\tilde\epsilon_0$ be sufficiently small such that for $0\le \epsilon\le \tilde\epsilon_0,$ we have
$\|\hat\omega\|\le \frac12 \min\{|\omega_{01}|, |\omega_{02}|\} $ and then $ |\frac{\omega_{*2}}{\omega_{02}}|\ge 1/2.$
Thus $\lambda(\epsilon)$ is continuous on $[0,\tilde\epsilon_0]$ with $\lambda(0)=0.$
Set $\lambda_0=\max_{[0,\tilde\epsilon_0]}|\lambda(\epsilon)|.$

 If $\lambda_0=0$,   for all  $\epsilon\in[0,\tilde\epsilon_0], $
  $\omega_*(\epsilon)=\frac{\omega_{*2}}{\omega_{02}}( \omega_{01}, \omega_{02})$. Therefore,
$$|\langle k, \omega_*(\epsilon)\rangle |\ge \frac{\alpha}{2|k|^{\tau}}, \ \forall k\in \Z^2\setminus\{0\}.$$

If $\lambda_0\ne 0$, without loss of generality, suppose $\lambda_0=\lambda(\epsilon_0)>0$  $(0<\epsilon_0\le \tilde\epsilon_0).$ Then we have $\{\lambda(\epsilon): \  \epsilon \in [0,\epsilon_0]\}\supset [0, \lambda_0].$
Let $\tilde\omega_*=( \omega_{01}+\lambda(\epsilon), \omega_{02}).$ Let $I^*_{\epsilon_0}=\lambda^{-1}(\Pi_{\lambda_0})$ be   the  inverse image of $\Pi_{\lambda_0}$
under the mapping $\lambda,$ where $\Pi_{\lambda_0}$ is defined as in  Lemma \ref{l7}. For $\epsilon\in\Pi_*$,
$\lambda(\epsilon)\in \Pi_{\lambda_0}$ and then
$$|\langle k, \omega_*(\epsilon)\rangle |=\bigl|\frac{\omega_{*2}}{\omega_{02}}\bigr| \cdot |\langle k, \tilde\omega_*(\epsilon)\rangle |\ge \frac{\alpha}{2|k|^{2\tau+2}}, \ \forall k\in \Z^2\setminus\{0\}.$$

 Recall that $\lambda$  maps $I^*_{\epsilon_0}$ onto $\Pi_{\lambda_0}.$ Then the set  $I^*_{\epsilon_0}$ has at least continuous carnality and so is non-empty.
  Moreover, when
  $\hat\omega(\epsilon)$ is differentiable, we have
  $$0< \mbox{meas}(\Pi_{\lambda_0})=\mbox{meas}(\lambda(I^*_{\epsilon_0}))\le c\cdot \mbox{meas}(I^*_{\epsilon_0}),$$
  which suggests $\Pi_{\lambda_0}$ has positive measure.
\hfill $\Box$
\end{proof}

\begin{lemma}\label{l1*}
Suppose the Brouwer degree of  the frequency mapping $\omega$ at
$\omega_0 $ on $\Pi$ is not vanishes, i.e.
$\mathrm{deg}\, (\omega, \Pi, \omega_0) \ne 0.$
Let $\hat\omega(\xi)$ and $\lambda(\xi)$ be  continuous on $\Pi.$
 Then  there exists a sufficiently small $\sigma_0>0$   such that  if $\|\hat\omega\| \le\sigma_0$ and  $\|\lambda\|\le\sigma_0$,
there exists at least one $\xi_*\in \Pi$ such that
$$ \omega_*(\xi_*)=\omega(\xi_*)+\hat\omega(\xi_*)=(1+\lambda(\xi_*))\omega_0.$$
\end{lemma}

\begin{proof}
 Let $\tilde\omega(\xi)=\hat\omega(\xi)-\lambda(\xi)\omega_0,$ then $\|\tilde\omega\|\le c\sigma$ with $c=1+|\omega_0|.$
The theory of   Brouwer    degree shows that, if $\sigma$ is sufficiently small,   $\mbox{deg}(\omega+\tilde\omega, \Pi, \omega_0) \ne 0.$
Thus the equation $\omega(\xi)+\tilde\omega(\xi)=\omega_0$ has at least one solution $\xi_*$ in $\Pi$. This proves the lemma.
\hfill $\Box$
\end{proof}

\begin{lemma}\label{l2*} Let $\omega(\xi)=(\omega_1(\xi), \ldots, \omega_{n-1}(\xi), \omega_{0n}) $  be continuous for $\xi\in \Pi\subset \R^{n-1}$,  where $\omega_{0n}$ is a constant.
Set
 $$\omega^{\flat}=(\omega_1, \ldots, \omega_{n-1}),\quad \hat\omega=(\hat{\omega}_1,   \ldots, \hat\omega_n),\quad \lambda =\hat\omega_{n}/\omega_{0n}.$$
Let $ \xi_0$ be an interior point in  $\Pi $, $  \omega^{\flat}_0=\omega^{\flat}(\xi_0)$ and  $\omega_0=(\omega^{\flat}_{0},~ \omega_{0n}) $.
Suppose$$\mathrm{deg}(\omega^{\flat}, \ \Pi, \  \omega^{\flat}_0) \ne 0,$$
then there exists a sufficiently small constant  $\sigma_0>0 $ such that  if $\|\hat\omega\|=\sigma<\sigma_0$,
 there exists $\xi_*\in\Pi$    such  that $\omega_*(\xi_*)=(1+\lambda(\xi_*))\omega_0.$ Moreover, $\|\lambda\|=O(\sigma/c).$


\end{lemma}

\begin{proof} Consider
the equation  $$\omega^{\flat}(\xi)+\hat\omega^{\flat}(\xi)=\bigl(1+\lambda(\xi)\bigr)
\omega_0^{\flat}.$$
Apply  Lemma \ref{l1*} to obtain that,  if $\sigma_0 $ is sufficiently small, the above equation has at least one solution $\xi_*\in \Pi$.
The definition of $\lambda$ yields  $\omega_{*n}=\omega_{0n}+\hat\omega_n= (1+\lambda )\omega_{0n}$. Then we have proved  the lemma.\hfill $\Box$
\end{proof}

\begin{prop}\label{l3*}
 Let   $\omega(\xi) \in \R^n$ and  $  \Omega(\xi)\in \R$ be continuous on $\Pi \subset \R^n$, and $|\Omega(\xi)|\ge c>0$ holds for all  $\xi\in \Pi$. Set  $\xi_0\in \Pi$ and $(\omega_0, \Omega_0)=(\omega(\xi_0), \Omega(\xi_0)).$
Suppose
$$\mathrm{deg}\, (\omega / \Omega,~ \Pi,~ \omega_0 / \Omega_0) \ne 0.$$ 
 Let $\lambda=\hat\Omega/ \Omega.$ Then there exists a sufficiently small $\sigma>0$ such that if   $\|\hat\omega\|+\|\hat\Omega\|\le \sigma,$
then there exists
$\xi_*\in\Pi$ such that
$$(\omega_*(\xi_*), \Omega_*(\xi_*))=  \Omega_0^{-1} \cdot \Omega(\xi_*) (1+\lambda(\xi_*))(\omega_0,\Omega_0).$$

\end{prop}

\begin{proof} Let $\tilde\omega=(  \omega/\Omega, 1)$ and $\hat{\tilde\omega}=(\hat\omega/ \Omega,~\hat\Omega/  \Omega).$ Then   $(\omega_*,\Omega_*)=\Omega\cdot \bigl(\tilde\omega+\hat{\tilde\omega}\bigr).$
Apply Lemma \ref{l2*} to  $\tilde\omega=(\tilde\omega^{\flat},~\tilde\omega_{n+1})$ with $\tilde\omega^{\flat}= \omega/\Omega$ and $\tilde\omega_{n+1}=1 $ to complete the proof.
\hfill $\Box$
\end{proof}

\begin{lemma}\label{l8}
Let   $O\subset \R^n$ be  an open connected bounded domain   and $\sigma>0 $.
Let $\omega_0\in O$, $\mu_0\ne 0$, $(\omega_0, \nu_0)\in \tilde O_{\alpha, \tau},$ where $\tilde O_{\alpha, \tau}$ is defined as in (\ref{dio*}) with $\bar n=1.$
 Let   $\hat\mu(\omega)$ be defined on $O$ and   $$f_k(\omega_0, \lambda)=\langle\omega_0, k\rangle +\nu_0- (1+\lambda)^{-1}\bigl( \lambda\cdot \mu_0- \hat\mu((1+\lambda)\omega_0)\bigr),$$
 where $\lambda$ is a small parameter.
  Denote by  $I_\sigma=[-\sigma, +\sigma]$   and
    $$I_{\sigma}^*= \left\{\lambda\in I_{\sigma}:  \  | f_k(\omega_0, \lambda) | \ge   \frac{\alpha}{2|k|^{2\tau+1}},~~ \forall k \in \Z^n\setminus \{0\}\right\}.$$
      Then there exists a sufficiently small $\sigma_0 ,$
   depending  on $\alpha$ and $\tau,$
   such that if  $\|\hat\mu\|_{C^1(\bar O)}=\sigma\le \sigma_0$,  $I_{\sigma}^*$ has positive measure with
  $\mbox{meas}(I_{\sigma}\setminus I_{\sigma}^*)=o(\sigma)$ as $\sigma\to 0.$

  \end{lemma}

   \begin{proof}
In view of $\mu_0\ne 0$,   there exists a sufficiently small  $\sigma_0>\sigma$  such that
$$| \partial_{\lambda}  f_k(\lambda,\omega_0) | \ge |\mu_0|/2 $$ holds for $|\lambda|\le \sigma_0$ and $\|\hat\mu\|_1\le \sigma $.

For  $ \lambda\in I_{ \sigma} ,$ we have
$$\bigl|(1+\lambda)^{-1}\cdot \lambda\cdot \mu_0 -(1+\lambda)^{-1}\cdot \hat\mu\bigl((1+\lambda)\omega_0\bigr)\bigr|\le c\sigma.$$
  Recall  that $|\langle\omega_0, k\rangle +\nu_0|  \ge \frac{\alpha}{|k|^{\tau}}.$
  If $c\sigma \le \frac{\alpha}{2|k|^{\tau}},$
  then $|f_k(\lambda,\omega_0)| \ge \frac{\alpha}{2|k|^{\tau}} $ holds.
  Thus,    we only need to consider the  case of   $c\sigma >  \frac{\alpha}{2|k|^{\tau}}.$

For each $k\in \Z^n\setminus \{0\}$, define
$$\Delta_k=\left\{\lambda\in I_{\sigma}: \  \  | f_k(\omega_0, \lambda) | <  \frac{\alpha}{2|k|^{2\tau+1}}\right\}.$$
Then $$\mathrm{meas}(\Delta_k)\le \frac{\alpha}{\mu_0|k|^{2\tau+1}}\le c\sigma
      \frac{1}{|k|^{\tau+1}}.$$
  Note that $\tau>n-1$ and  $$I_{\sigma}\setminus I_{\sigma }^*\subset \bigcup\limits_{  2c \sigma  |k|^{\tau}> \alpha  }\Delta_k.$$
  Therefore,
  $$\mbox{meas}(I_{\sigma}\setminus I_{\sigma }^*)\le \sum\limits_{   2c \sigma  |k|^{\tau}> \alpha}\mbox{meas}(\Delta_k)\le c\sigma \bigl(\frac{\sigma}{\alpha}\bigr)^{\frac{\tau-n+1}{\tau}}.$$
When  $\sigma_0$ is sufficiently small such that $c(\frac{\sigma_0}{\alpha})^{\frac{\tau-n+1}{\tau}}<1$,    $I_{\sigma}^*$  is non-empty with  positive measure.
\hfill $\Box$
\end{proof}



\begin{prop}\label{l9}\

\vskip 0.2 true cm\noindent
(1) Let $(\omega_0, \Omega_0)\in \tilde O_{\alpha, \tau}$ with $\bar n=1,$ and
\begin{equation}
I_{\sigma}^{*}=\{\lambda\in I_{\sigma}:~((1+\lambda)\omega_0, \Omega_0)\in \tilde O_{\alpha/2,2\tau+1}\}.
\end{equation}
If $\sigma $ is sufficiently small, then $I_{\sigma }^{*}$ is   non-empty  and satisfies
 $$\mbox{meas}(I_{\sigma}\setminus I_{\sigma}^*)\le c\sigma\bigl(\frac{\sigma}{\alpha}\bigr)^{\frac{\tau-n+1}{\tau}}.$$

\vskip 0.2 true cm \noindent
(2) Let $\omega(\xi)\in C(\Pi)$ and   $\omega_0=\omega(\xi_0).  $ Suppose   $\mathrm{deg}(\omega, \Pi, \omega_0) \ne 0.$
Let $\omega_*=\omega+\hat \omega$ and $\Omega_*=\Omega_0+\hat\Omega.$
   Then, there exists a sufficiently small constant $\sigma_0>0$, such that if
 $\|\hat\omega\|+\|\hat\Omega\|=\sigma\le\sigma_0,$
\,  for $ \lambda\in I_{\sigma }^*$
  there exists $\xi_*\in \Pi$ such that
 $$(\omega_*(\xi_*),\Omega_*(\xi_*)) =(1+\hat\Omega(\xi_*)/{\Omega_0})\cdot \bigl((1+\lambda)\omega_0, \Omega_0\bigr).$$
  \end{prop}

\begin{proof}
The first  conclusion  follows directly from  Lemma \ref{l8}.
Now we prove the second one.

Rewrite as
$$(\omega_*,\Omega_*)=(1+\hat\Omega/\Omega_0)\cdot (\omega+\tilde\omega,\Omega_0)$$
where
$$\tilde\omega=(\Omega_0+\hat\Omega)^{-1}\cdot (\Omega_0\cdot\hat\omega- \hat\Omega \cdot \omega).$$
Observe that $\|\tilde\omega\| \le c\sigma.$   Lemma \ref{l1*} shows, there exists a  sufficiently small $\sigma_0 $, such that  if $\|\hat\omega\|+\|\hat\Omega\|=\sigma\le\sigma_0$ and $|\lambda|\le \sigma $,
we have
$$\mathrm{deg}\bigl( \omega+\tilde\omega,~ \Pi, ~ (1+\lambda)\omega_0\bigr)\ne 0.$$
Thus, for $\lambda\in I_{\sigma^*} $ there exists   $\xi_*\in\Pi$ such that   $\omega(\xi_*)+\tilde\omega(\xi_*)=(1+\lambda)\omega_0.$
 \hfill $\Box$


\end{proof}

\begin{prop}\label{l8*}
Let   $O\subset \R^n$ be an open bounded connected  domain, and  $$\Omega(\omega)=\beta+\omega \cdot M, \ \ \omega\in O, $$ where $\beta_{0}=(\beta_{ 1}, \ldots, \beta_{ \bar n})$ and $M$ is an $ n\times \bar n$ constant matrix.
Let $\omega_0\in O$ and $ \Omega_0=\Omega(\omega_0).$ Suppose
 $(\omega_0,  \Omega_0)\in \tilde O_{\alpha, \tau}$  and $\beta$ satisfies
$$\la l, \beta\ra \ne 0, \quad \forall ~l\in \mathcal{L}.$$
Set $\hat\omega(\omega), \hat\Omega(\omega)\in C^1(O)$ and
 $$\bigl(\omega_*(\omega), \, \Omega_*(\omega)\bigr)=\bigl(\omega, \, \Omega \bigr)+\bigl(\hat\omega,\, \hat\Omega \bigr).$$
Then, there exists a sufficiently small $\sigma_0 $, such that if $\|\hat\omega\|_1+\|\hat\Omega\|_1=\sigma\le \sigma_0$,
there exists  a non-empty subset
$I_{\sigma}^*\subset   I_{\sigma}$
with the estimate
 $$
\mbox{meas}(I_{\sigma}\setminus I_{\sigma}^*)=o(\sigma).$$
  Moreover,  for any  $\lambda\in I_{\sigma}^*$ ,   there exists  $\varpi\in O$ such that
$$ \omega_*(\varpi)=(1+\lambda)\omega_0 \  \ \mbox{and}\ \    \bigl(\omega_*(\varpi), \Omega_*(\varpi)\bigr)\in \tilde O_{\alpha/4,2\tau+1} .$$

\end{prop}

\begin{proof}
At first we note that if $\sigma$ is sufficiently small, $\omega_*$ is also  non-degenerate in $\omega$ on $O$. So without loss of generality, we assume $\omega_+=\omega_*(\omega)$ as independent parameter. The inverse $\omega=\omega(\omega_+)$ is well defined  in a little smaller domain $O_+\subset O$.
Then $\hat\omega\circ\omega(\omega_+)$ and $\hat\Omega\circ\omega(\omega_+)$ depend on $\omega_+$ and satisfy $\|\hat\omega\|_1+\|\hat\Omega\|_1\le c\sigma$
on $O_+$.
Thus $\omega=\omega_+-\hat\omega(\omega_+)$ and $\Omega\circ\omega(\omega_+)=\beta+\omega_+\cdot M-\hat\omega(\omega_+)\cdot M.$
Then we have
 $$ \Omega_*\circ\omega(\omega_+)=\beta+\omega_+\cdot M+\tilde\Omega(\omega_+),\ \  \tilde\Omega(\omega_+)=-\hat\omega\circ\omega(\omega_+)\cdot M +\hat\Omega\circ\omega(\omega_+).$$
Let $\omega_+=(1+\lambda)\omega_0$ and $\varpi(\lambda)=\omega((1+\lambda)\omega_0).$
Then $\varpi(\lambda)$ is well defined  for sufficiently small $\lambda$.
 Then we consider

$$
\bigl(\omega_{*}\circ \varpi(\lambda), \  \Omega_{*}\circ \varpi(\lambda)\bigr)=\bigl((1+\lambda)\omega_0,
 \Omega_0+\lambda\omega_0\cdot M+\tilde\Omega((1+\lambda)\omega_0)\bigr)
$$
Rewrite as
\begin{equation}\label{ZT}
(1+\lambda)^{-1} \bigl(\omega_{*}\circ \varpi(\lambda), \ \Omega_{*}\circ \varpi(\lambda)\bigr)=\bigl( \omega_0, \Omega_0 -(1+\lambda)^{-1}(\lambda\cdot \beta
-\tilde{\Omega}(\lambda)\bigr),
\end{equation}
where $$\tilde{\Omega}(\lambda)= \tilde\Omega((1+\lambda)\omega_0).$$

To apply Lemma \ref{l8}, for each fixed $l\in \mathcal{L}$, let   $\nu_0=\la l,  \Omega_0\ra,$ $\mu_0=\la l,\beta \ra, \hat\mu=\la l, \tilde\Omega \ra $. Then there exists a sufficiently small  $\sigma_0>0$  such that for $\sigma\le \sigma_0$, we have
$ I_{\sigma}^{ l*}\subset I_{\sigma} $ with the estimate
$$\mbox{meas}(I_{\sigma}\setminus I_{\sigma}^{l*})=o(\sigma),$$
  such that for each $l\in \mathcal{L}$ and $\lambda\in I_{\sigma}^{ l*},$

\begin{equation}\label{ZT*}
\bigl| \la \omega_0, k\ra + \nu_0-(1+\lambda)^{-1} (\lambda\cdot \mu_0-
\hat{\mu}) \bigr|\geq \frac{\alpha}{2|k|^{ 2\tau+1}},\quad k\in \Z^n\setminus \{0\}.
 \end{equation}

Define $$I_{\sigma}^{*}= \bigcap\limits_{l\in \mathcal{L}} I_{\sigma}^{l*}.$$
Recalling   $\mathcal{L}=\{l\in \Z^{\bar{n}}:~1\leq |l|\leq 2\}$,  we arrive at
 $$ \mbox{meas}(I_{\sigma}\setminus I_{\sigma}^{*})=o(\sigma).$$

Moreover,  in view of the definition (\ref{dio*}), the assumption $(\omega_0,  \Omega_0)\in \tilde O_{\alpha, \tau} $ shows
 $$|\la l,  \Omega_0\ra |\ge \alpha \quad\text{for}~l\in \mathcal{L}.$$
 Combining  $\lambda\in I_{\sigma}$ with   $\|\hat\omega\|_1+\|\hat\Omega\|_1=\sigma\le \sigma_0$, for sufficiently small $\sigma$, we have
\begin{equation} \label{ZT0}
|\la l, (1+\lambda)^{-1}\Omega_{*}\circ \varpi(\lambda)\ra |= |\la l, \Omega_0 -(1+\lambda)^{-1}(\lambda\cdot \beta
-\tilde{\Omega}(\lambda))\ra |\ge \alpha/2 ~~\text{for}~l\in \mathcal{L}.
\end{equation}
Summarizing the above estimates (\ref{ZT*}) and (\ref{ZT0}), it follows that  for $\lambda\in I_{\sigma}^{*}$,
$$\bigl(\, \omega_0,  \ (1+\lambda)^{-1}\Omega_{*}\circ \varpi(\lambda) \, \bigr) \in \tilde O_{\alpha/2, 2\tau+1}.$$
If $\sigma_0\le \frac12,$ then $$\bigl(\, \omega_*\circ \varpi(\lambda), \   \Omega_{*}\circ \varpi(\lambda) \, \bigr) \in \tilde O_{\alpha/4, 2\tau+1}.$$
Note that $ \omega_*( \varpi)=(1+\lambda)\omega_0$.
Thus we prove this proposition.
\hfill $\Box$


\end{proof}

\section{ Application of Theorems \label{c*}}

In this section, by virtue of the previous discussion on the stability of Diophantine frequencies, our Theorems \ref{TT2} and \ref{TT3} can be applied
to various situations and obtain interesting results, some of which have been displayed in the literature; while some are rather novel.
This wide application accounts for the advantage of our theorems.

\vskip 0.3 true cm

$\bullet$  {\it The classical KAM theorem.}
\vskip 0.2 true cm

We first point that  the Kolmogorov
    non-degeneracy condition and  R\"ussmann's non-degeneracy condition  are stable under small perturbation.
    Thus, by standard measure estimate, for most of parameter $\xi$,   $\omega_{*}(\xi)$ belongs to the Diophantine set $O_{\alpha,\tau}$. Then Theorem   \ref{TT2} immediately shows, $H$ possesses an invariant torus with the frequencies $\omega_{*}(\xi)$, as is stated     in \cite{P1,E1,R1,R2,X1}.

\vskip 0.3 true cm
$\bullet$  {\it KAM tori with prescribed frequency.}
\vskip 0.2 true cm


We indicate that the result in \cite{Xu4}   follows obviously from Theorem \ref{TT2} and Lemma \ref{l1*}. However, due to the method of introducing external parameter,   \cite{Xu4} only presents the  existence of invariant tori with one single  prescribed frequency vector, and fails to
 obtain the  smoothness of invariant tori with respect to the parameter.
However,   Theorem \ref{TT2} can tell not only the existence of invariant tori, but also the $C^m$-smoothness   in  the parameter.
 In fact, the parameterized Diophantine frequencies in $\omega_*(\Pi)$ are   $C^m$-smooth w.r.t. $\xi$, and so are the corresponding invariant tori.

In particular, by the  theory of topological degree, our theorems  can  apply to some hamiltonian that only  continuously depends on the parameter. See the following instance.

\vskip 0.2 true cm
Consider the hamiltonian (\ref{LH1}) with
$$\omega(\xi)=(\xi_1^{2l_1+1}, \ldots, \xi_n^{2l_n+1}), \ \ \Pi=\{\xi: \  |\xi_i|\le 1, i=1, \ldots, n\},$$
where $l_i\ge 0$ are integers.
Let $0<\alpha<1. $ If $\epsilon$ is small, the  theory of topological degree implies
$$  \omega_*(\Pi)\supset O=\{\omega=(\omega_1,\cdots \omega_n)\in \R^n: \  |\omega_i|\le (1-\alpha)^{2l_i+1},~ i=1, \ldots n\}.$$
 Note that $O$ is also a domain.  Thus,  for the parameterized hamiltonian $H(\xi; \theta,I)$,    all the invariant tori with frequencies in $O\cap O_{\alpha,\tau}$   persist. Moreover, these invariant tori depend  on the parameter
 $C^m$-smoothly  in Whitney's sense \cite{W}.


\vskip 0.3 true cm
$\bullet$  {\it  KAM theorem with  Bruno  non-degeneracy condition.}
\vskip 0.2 true cm



Consider the hamiltonian  (\ref{LH1}) and $\omega(\xi)$ satisfies  the Bruno  non-degeneracy condition  (\ref{BR}).
 Proposition \ref{l6} illustrates,
  for any $\omega_0=\omega(\xi_0)\in O_\alpha$,  there exists an one-parameter continuous  family of
 invariant tori with the frequencies $(1+\lambda(\eta))\omega_0,$  where the parameter $\eta$  is close to  zero and $\lambda=O(|\eta|+\sigma)$ with $\sigma=\frac{\epsilon}{2r\alpha}.$
Especially, when  the hamiltonian  depends  on the parameter analytically,   the obtained  family can  be proved  analytically dependent on $\eta$ near zero.

\vskip 0.3 true cm
$\bullet$  {\it  KAM theorem for  hamiltonian system with two degrees of freedom.}
\vskip 0.2 true cm


Let $H(\epsilon;\theta, I)=
\langle\omega_0, I\rangle + \epsilon P(\epsilon;\theta, I),$ where    $P$ is  real analytic in $(\theta, I)$
on $ D(s,r)\subset \C^2\times \C^2,$ and $C^m$-smooth in  a small parameter  $\epsilon$ on $I_{\epsilon_0}=[0, \epsilon_0]$.
 Suppose $\omega_0=(\omega_{01},\omega_{02})\in O_{\alpha,\tau}.$ Applying  Theorem \ref{TT2} and Proposition \ref{l7*}, we have the following results:

 There exists a sufficiently small constant $\epsilon_0>0,$
 such that
   if   $$\|\epsilon P\|_{\alpha, \Pi\times D(s,\,r)}=\epsilon \le \alpha r s^{\tau'} \gamma =\epsilon_0,$$ where $\tau'=n+(m+1)(2\tau+2)+m,$
 there always  exists an non-empty set $I_{\epsilon_0}^* \subset I_{\epsilon_0} $ such that for $\epsilon\in I_{\epsilon_0}^*$,   $H(\epsilon;\theta, I)$ has
  invariant tori with  frequencies $\omega_*(\epsilon)=\omega_0+\hat\omega_0(\epsilon)\in O_{\frac{\alpha}{2}, 2\tau+2}$ satisfying  $|\hat\omega_0(\epsilon)|\le  2\epsilon/r.$
  Moreover,
  for  $m=0$,   $I^*_{\epsilon_0}$ has  continuous cardinality; for  $m\ge 1$,   $I^*_{\epsilon_0}$ has  positive measure.

\vskip 0.2 true cm
  The above result implies that  the invariant tori with Diophantine frequencies  for an integrable hamiltonian with two degrees of freedom    never isolate, which was  pointed and proved by Elliasson in \cite{E2}.

Note that here we do not require analytic condition of the hamiltonian in the parameter; therefore, we cannot obtain an accurate measure estimate for $I_{\epsilon_0}$.
 In  \cite{Xu2}, the authors   considered the same problem for analytic hamiltonian  in both the phase variables $(\theta,I)$ and the small parameter $\epsilon.$  Without imposing
any non-degeneracy condition in advance, the authors  obtained a similar result   with
$\mbox{meas}(I_{\epsilon_0}\setminus I^*_{\epsilon_0})=o(\epsilon_0)$ as $\epsilon_0\to 0 .$



\vskip 0.3 true cm
$\bullet$  {\it  Elliptic lower dimensional  KAM-tori.}
\vskip 0.2 true cm

1. Case of one normal dimension:
\vskip 0.2 true cm
Consider  the hamiltonian  (\ref{WH1}) with $\bar n=1$ and $\Omega(\xi)\equiv \Omega_0.$
Suppose   $(\omega_0, \Omega_0)=(\omega(\xi_0),\Omega_0)\in \tilde O_{\alpha,\tau}$ and
  $\omega(\xi)$ satisfies $\mbox{deg}(\omega, \Pi, \omega_0) \ne 0$.
 By Proposition \ref{l9} and Theorem \ref{TT3},
  there exist   sufficiently small constants
  $\gamma>0$ and $\sigma_0>0$ such that  if
  $$\|P\|_{\Pi;D(s,r)}=\epsilon \le \alpha r^2 s^{\tau'}\gamma \, \mbox{ with} \,
  \tau'=n+(m+1)(2\tau+1)+m,$$
  and $\sigma=\epsilon/2r^2\le \sigma_0$,
  there exists   $I^{*}_\sigma\subset I_{\sigma}$  with $\mbox{meas}(I_{\sigma }\setminus I^{*}_\sigma) =o(\sigma)$ as $\sigma\to 0$, such that
 for all $\lambda \in I^{*}_\sigma$ there exist $\xi_*\in \Pi$ and $\tilde\lambda=\hat\Omega(\xi_*)/\Omega_0$ with $|\tilde\lambda|\le \sigma/|\Omega_0|$,  such that the hamiltonian $H(\xi_* , \cdot)$ has
an invariant torus with tangential frequency $(1+\tilde\lambda)(1+\lambda)\omega_0$ and normal frequency $(1+\tilde\lambda)\Omega_0.$

\begin{rem}
Proposition \ref{l3*} and Theorem \ref{TT3}  can also be applied to   $H(\xi; \theta, I, z, \bar z)$   with $\bar n=1$ and $\Omega=\Omega(\xi).$
Let  $(\omega_0, \Omega_0)=(\omega(\xi_0),\Omega (\xi_0))\in \tilde O_{\alpha,\tau}$ and
 suppose $\mathrm{deg}\, (\omega / \Omega,~ \Pi,~ \omega_0 / \Omega_0) \ne 0.$ Then we can arrive at  an analogous    result.
\end{rem}


\vskip 0.2 true cm
2. Case of multiple normal dimensions:
\vskip 0.2 true cm

Consider the hamiltonian
$$H(\omega; I, \theta, z, \bar z)=\langle\omega, I\rangle +  \langle\Omega(\omega), z\, \bar z\rangle+P(\omega;\theta, I, z,\bar z)$$
 as in Theorem \ref{TT3} with $m\ge 1$, where the parameter $\omega\in  O\subset \R^n$.
The normal frequency vector is $$\Omega(\omega)=\beta+\omega \cdot M, \ \ \omega\in O, $$ where $\beta=(\beta_{1}, \ldots, \beta_{\bar n})$ and $M$ is an $ n\times \bar n$ constant matrix.

Suppose $\la l,\beta\ra \neq 0~~\text{for}~l\in \mathcal{L}.$ Define
$$O_{*}=\{\omega\in O:~(\omega, \, \Omega(\omega))\in \tilde O_{\alpha, \tau}\}.$$
Then we can verify that $O_{*}$ occupies a large portion of measure in $O$ for sufficiently small constant  $\alpha>0$.


Set $\omega_0\in O_{*}$ and $ \Omega_0=\Omega(\omega_0).$
  Then the combination of Proposition \ref{l8*} and Theorem \ref{TT3} yields, there exist    sufficiently small constants  $\gamma $ and $\sigma_0 $
   such that
  if
  $$\|P\|_{\bar O; D(s,r)}=\epsilon \le \frac{\alpha}{4} r^2 s^{\tau'}\gamma \, \mbox{ with}\,
 \tau'=n+(m+1)(2\tau+1)+m,$$ and $\sigma=\frac{\epsilon}{2r^2\alpha}\le \sigma_0,$
   there exists an non-empty Cantor subset $I^*_\sigma \subset   I_{\sigma}$
 and for  $\lambda \in I^*_\sigma$ there exists $\varpi\in O$ such that the hamiltonian $H(\varpi, \cdot)$ has
an invariant torus with frequencies $\bigl(\omega_{*}(\varpi),~\Omega_{*}(\varpi)\bigr)=\bigl((1+\lambda)\omega_0,~\Omega_{*}(\varpi)\bigr).$ Moreover,  we have
  $\mbox{meas}(I_{\sigma}\setminus I^*_\sigma)=o(\sigma)$ as $\sigma \to 0 .$

 \vskip 0.2 true cm
 In the case of $ M=0$,   the above result implies  that  obtained by Bourgain in \cite{Bou2}. We indicate that our assumption is a little stronger than  in \cite{Bou2}, where only  the first Melnikov's condition is required. Nevertheless,  under the second  Melnikov  condition,  we can obtain  the normal form for the persisting invariant tori, which provides the linear stability of these invariant tori and reveals more dynamical information.

 Note that by  some asymptotic property of the normal frequencies, Proposition \ref{l8*} can be extended to some infinite dimensional hamiltonian as showed in \cite{Ber}.



\section{Proof of Theorems }\label{d}
\setcounter{equation}{0}

In this section, we mainly prove Theorem \ref{TT2} and omit the proof of  Theorem \ref{TT3} since the idea is the same only with some  modified  KAM estimates.   Our proof is based on a KAM iteration. The key is to choose a suitable constant $\alpha$ in the small divisor conditions.  Usually the constant $\alpha$  decreases as the KAM step proceeds; here it will be   increasing.
Moreover, we shall present an explicit  extension of small divisors  rather than using  Whitney's extension theorem\cite{W}. In particular, even though small divisor condition does not hold, our extension  still  works, which plays
an important role in separating  the KAM iteration and non-degeneracy condition. We should note that the idea of the small divisor extension is also  used by Elliasson in \cite{E2}. In fact,  the spirit in our proof is more or less similar to that in \cite{E2}. More precisely, the existence of KAM tori depend on existence of Diophantine frequencies in the final KAM step ( the limit of KAM iteration).

\vskip 0.2 true cm {\bf KAM-step.}
 We summarize our KAM step
 in the following iteration lemma.

 \vskip 0.2 true cm
\begin{lemma}{\rm (Iteration Lemma)}  \label{T3} Consider the following
hamiltonian $$H(\xi;\theta,I ) = N(\xi;I) +
P(\xi;\theta, I),$$ where $N(\xi
;I)=\langle\omega(\xi), I \rangle.$
 Let  $\alpha\le \alpha_*\le 2\alpha $, $\tau>n-1, m\ge 0, \tau'=n+m+\tau (m+1).$ Assume $\omega\in C^m(\Pi_0)$ and
 $$\|P\|_{\alpha_*, \Pi_0\times D(s,r)}\le \epsilon=\alpha r\rho^{\tau'} E.$$
Set $
s_+=s-5\rho, \;\eta=\sqrt{E}, \;  r_+=\eta r.$
Then the following conclusions hold:\\
(i)~ For any $\xi\in
\Pi_0$ there exists a symplectic mapping
$$\Phi(\xi;\cdot ,\, \cdot): D(s_+, \, r_+)\to D(s,r),$$
which is real analytic in $(I,\theta)$ on $D(s_+, r_+)$ and $C^m$-smooth in $\xi$ on $\Pi_0 $
  such that
  $$\|W(\Phi-id)\|_{\alpha_*, \Pi\times D(s_+, r_+)},~~\|W({\cal D}\Phi-Id)W^{-1}\|_{\alpha_*,\Pi_0\times D(s_+, r_+)
}\le c
E,\; $$  where $\cal D$ is the differentiation operator with
respect to  $(\theta, I) $ and  $ W=\mbox{diag}(\rho^{-1} I_n,
r^{-1}I_n)$ with $I_n$ being the $n$-th  unit matrix.\\
(ii) There exists a real analytic hamiltonian  $$H_+(\xi;I,\theta)=N_+(\xi; I) + P_+(\xi;\theta, I) $$
   defined on $D(s_+, r_+)$, that is  $C^m$-smooth in $\xi\in \Pi_0,$ where
 $$N_+(\xi;I)=\langle\omega_+(\xi), I \rangle,\quad \omega_+=\omega+\hat\omega$$ with the estimate
 $$\|\hat\omega\|_{\alpha_*, C^m(\Pi_0)}\le \epsilon/r.$$
 $P_+$ denotes the new perturbation  satisfying
$$\|P_+\|_{\alpha_*, \Pi_0\times D(s_+,r_+)}\le \epsilon_+=\alpha_+ r_+
\rho_+^{\tau'}E_+.$$
Here,
$$\rho_+=\frac12\rho, ~~  E_+=c(m,n,\tau)\cdot E^{\frac32},~~  \alpha \le \alpha_+\le 2\alpha.$$
\noindent
(iii)
Set $e^{-K \rho}=E$ and   $$O_\alpha^K=\bigl\{\omega\in\R^n :\ |\langle\omega, k\rangle|\ge \frac{\alpha}{|k|^{\tau}}, \  0<|k|\le K.\bigr\}  $$
Suppose $2 K^{\tau+1} \epsilon  \le  (\alpha_+-\alpha) r$ and define
\begin{equation}\label{FG}
\Pi=\bigl\{\xi\in\Pi_0 :\ \ \omega(\xi)\in O_\alpha^K \bigr\}, \quad  \Pi_+=\bigl\{\xi\in\Pi_0:\ \ \omega_+(\xi)\in O_{\alpha_+}^{K_+} \bigr\},
\end{equation}
where $K_+>K$   satisfies  $e^{-K_+ \rho_+}=E_+.$
Then we have  $\Pi_+\subset \Pi$.  \\
 Moreover,
$$H\circ \Phi( \xi;\theta, I)=H_+(\xi;\theta, I)= N_+(\xi; I) + P_+(\xi;\theta, I),\quad \text{for}~~\xi\in\Pi .$$

\end{lemma}

\vskip 0.2 true cm \noindent
 {\bf Proof of Iteration Lemma.}
  Our KAM step is  standard and  we  divide it into  several parts.
Here and below we use $c$ to indicate the constants which are independent of KAM steps.

 \vskip 0.2 true cm
 {\it A. Truncation.} Set $R=P(\xi;\theta, 0) + \langle
P_I(\xi;\theta, 0), I \rangle .$ It follows easily that
$\|R\|_{\Pi_0\times D(s,\,r)}\le 2\|P\|_{\Pi_0\times D(s,\,r)}\le
2\epsilon.$  Let
 $$R=\sum_{k\in \Z^n}R_k(\xi;I)e^{\rm{i}\langle k,\,\theta \rangle }$$ and
 $$R^K=\sum_{|k|\le K}R_k(\xi;I)e^{\rm{i}\langle k,\,\theta \rangle }.$$
Then
 $$\|R-R^K\|_{\Pi_0\times D(s-\rho,\, r)}\le 2\epsilon e^{-K\rho}.$$

\vskip 0.2 true cm
 {\it B. Construction of symplectic mapping.} The
symplectic mapping is generated by a hamiltonian flow mapping at
$1$-time, that is, $\Phi=X_F^t|_{t=1},$ where $F$ is the generation
function. It follows that
 $$H \circ\Phi = N_+ + \{N,F\} + R^K-[R] + P_+,$$ where $[R]$
denotes the average of $R$ on $\T^n$ and $\{\cdot, \cdot \}$   the Poisson
bracket. The new normal form is   $N_+=N+[R]=\langle I, \omega_+(\xi)\rangle,$ $\omega_+=\omega+\hat\omega$ with $\hat\omega=\partial_I[R].$
  $$P_+= \int_0^1\{(1-t)\{N, F\}+R^K, F\}\circ
X^t_F\,dt + (P-R^K)\circ\Phi.$$

 We choose  $F$ such that
\begin{equation}
\{N,F\} + R^K-[R]=0.\label{l1}\end{equation}
Let $\{F_k\}$ and $\{R_k\}$ be relevant Fourier coefficients
with respect to $\theta$. Thus, $F_k=0$ with $k=0$ or $|k|>K;$ and  for $\langle\omega(\xi), k\rangle \ne 0,$
$$ F_k=\frac{1}{\rm{i}\langle\omega(\xi), k\rangle }R_k,\quad 0<|k|\le K.   $$
Thus, it follows
 $$P_+= \int_0^1\{(1-t)[R]+tR^K, F\}\circ
X^t_F\,dt + (P-R^K)\circ\Phi.$$

 \vskip 0.3 true cm
 \noindent
 {\it C. Extension of small divisors.}  Now we define a $C^{\infty}(\R)$-smooth
function $\varphi(t)$ as
\begin{equation}\nonumber
 \varphi(t)= \left \{ \begin{aligned}  & ~~ 0, \ \ \  |t| \le \frac12,   \\
&   ~~1, \ \ \  |t| \ge  1. \end{aligned}\right.
 \end{equation}
For $h>0$, let $\varphi_h(t)= \varphi(t/h)$. Then
$\varphi_h(t)\in C^{\infty}(\R)$ with


\begin{equation}
\label{e} |\frac{{d}^{ \ell} \ }{d t^{\ell}}\varphi_h(t)|\le c_{\ell}/
h^{l}, \quad \forall t\in  \R, \ \ \forall \ell\ge
1,\end{equation}  where $c_{\ell}$ is a constant depending on $\ell$.

Let
$$h=\frac{\alpha}{|k|^{\tau}},\   \ t_k(\xi)=\langle\omega(\xi), k\rangle , \ \ g_k(\xi)=\frac{\varphi_h(t_k(\xi))}{\rm{i}\langle\omega(\xi), k\rangle }.$$
Recall the definition of $\Pi$ in (\ref{FG}). Then  $g_k(\xi)=\frac{1}{\rm{i}\langle\omega(\xi), k\rangle } $ for $\xi\in \Pi$.
Note that even if $\Pi=\emptyset,$ the extension of  $g_k(\xi)$ is still well defined on $\Pi_0.$
Furthermore,  $g_k(\xi)\in C^m(\Pi_0)$ with the estimate
$$\bigl|\frac{\partial^\beta g_k}{\partial\xi^{\beta}}(\xi) \bigr|\le ch^{-|\beta|-1}|k|^{|\beta|},  \quad   \xi \in \Pi_0, \ \forall~ |\beta|\le m.$$

 Now we extend $F_k$   from $\Pi$ to the whole set $\Pi_0$ by setting
$$   \widetilde{F}_k(\xi;I)=g_k(\xi)R_k(\xi;I)=\frac{\varphi_h(t_k(\xi))}{\rm{i}\langle\omega(\xi), k\rangle }R_k(\xi;I),\,\, 0<|k|\le K.$$

Let $ \widetilde{F}(\xi;I,\theta)=\sum_{0<|k|\le K} \widetilde{F}_k(\xi;I) e^{{\rm i} \langle k,\theta\rangle} $ and
we have
$$\| \widetilde{F}\|_{\alpha_*,\Pi_0\times D(
r,s-\rho)}\le\frac{c\epsilon}{\alpha \rho^{\tau'}}, \ \ \tau'=n+\tau (m+1)+m.$$

\vskip 0.2 true cm
 {\it D. Estimates for symplectic mapping.}
 It follows from Cauchy estimate that
$$\|WX_{ \widetilde{F} }\|_{\alpha_*, \Pi_0\times D(
r,s-2\rho)}\le \frac{c\epsilon}{\alpha r \rho^{\tau'}}=c E,$$
where $W=\mbox{diag}(\rho^{-1}I_n, r^{-1}I_n).$

Thus, if $0< \eta\le \frac18$ and $c E \le \frac18, $   for all
$\xi\in \Pi$ we have
$$\Phi(\xi;\cdot,\cdot)=X_{ \widetilde{F} }^1: D(r\eta, s-3\rho) \to D(2r\eta, s-2\rho).$$
 Cauchy  estimate again yields
$$\|W(\Phi-id)\|_{\alpha_*, \Pi_0\times D(s-5\rho,\eta r)},~\|W({\cal D}\Phi-Id)W^{-1}\|_{\alpha_*, \Pi_0\times D(s-5\rho,\eta r)
}\le c E. $$

  \vskip 0.2 true cm
{\it E. New error terms.}
 Following the same approach as in the classical KAM theorem, we arrive at
$$\|P_+\|_{\alpha_*, \Pi_0\times D(s_+,r_+)}< c  \frac{\epsilon^2}{\alpha r \rho^{\tau'}}
+c (\eta^2+e^{-K\rho})\epsilon,  $$ where  $c$  is a constant depending only on $n$ and $\tau.$  The choice of the
parameters  $\eta$ and $K$ implies,
$$\|P_+\|_{\alpha_*, \Pi_0\times D(s_+,r_+)}\le c\epsilon E\le \alpha_+  r_+ \rho_+^{\tau'}  E_+= \epsilon_+.$$
 where   $ \alpha_+, \ \rho_+, \ r_+,  E_+$ are   given as  in the lemma.

Recall that $\hat\omega=\partial_I [R]$ and   we have  $\|\hat\omega\|_{\alpha_*, C^m(\Pi_0)}\le  \epsilon/r.$
Suppose $2 K^{\tau+1} \epsilon  \le  (\alpha_+-\alpha) r$, and then $\Pi_+\subset\Pi $ holds.

 \vskip 0.3 true cm
{\bf Iteration.} Now we choose some suitable sequences of parameters so that the
above step can iterate infinitely.

At the initial step, let  $\rho_0=s/20,$ $r_0=r,$ $E_0=2\cdot 20^{\tau'}\gamma>0$, $\alpha_0=(1-\frac12)\alpha $ and     $\epsilon_0=E_0\alpha_0 r_0 \rho_0^{\tau'}.$ Let  $\eta_0={E_0}^{\frac12}$ and
   $e^{-K_0\rho_0}=E_0.$

For $j\geq 0$, we define
\begin{align*}
\begin{gathered}
\rho_{j+1}=\frac12\rho_j,~r_{j+1}=\eta_j
r_j,~E_{j+1}= c E_{j}^{\frac32},~\alpha_{j+1}=(1-\frac{1}{2^{j+3}})\alpha.\\
\epsilon_j=E_j\alpha_j r_j \rho_j^{\tau'},~~~\eta_j=E_j^{\frac12},~~~e^{-K_j\rho_j}=E_j.
\end{gathered}
\end{align*}
 Note that
$\alpha_j\le \alpha\le 2\alpha_j.$
It is easy to verify $ cE_{j}\le ( cE_0)^{(\frac32)^j}$.

 Now we check the assumption  $2 K_j^{\tau+1} \epsilon_j  \le  (\alpha_{j+1}-\alpha_j) r_j.$
   This is equivalent to prove $F_j=2^{j+3}K_j^{\tau+1}\epsilon_j /r_j\le \alpha.$
 Notice that $$\frac{F_{j+1}}{F_j}=2cE_j^{\frac12}\cdot \biggl(\frac{K_{j+1}}{K_j}\biggr)^{\tau+1}.$$
It follows from  $K_j= -\ln E_j/\rho_j $  that
 $$ K_{j+1}/K_j=  2\ln  E_{j+1}/\ln E_j= (2\ln   \tilde c +3\ln E_{j} )/\ln  E_j \le 3.$$
 Then we have   $F_{j+1}\le  c E_j^{\frac12}F_j.$
 Note that $$ F_0=4\epsilon_0 K_0^{\tau+1}/r_0= 2 \cdot 20^{1-\tau' } s^{\tau'-1}E_0(\ln|E_0|)^{\tau+1} \alpha ,$$
which implies  for all fixed $s,r>0$ and  sufficiently small $E_0$,   $F_k\le \alpha$ holds for any $ k\ge  0.$
Hence we immediately derive $\Pi_{j+1}\subset \Pi_j$ from  the assumption   $2 K_j^{\tau+1} \epsilon_j  \le  (\alpha_{j+1}-\alpha_j) r_j $.

 Let $\Pi_0=\Pi$ and $D_j=D(s_j,r_j).$  Applying
    Iteration Lemma \ref{T3}, we have a sequence of monotonously decreasing closed sets $\{\Pi_j\}$,
  and a sequence of symplectic mappings $\{\Phi_j\}$
 such that for each $\xi\in \Pi,$
 $\Phi_j(\xi;\cdot,\cdot): D_{j+1} \to D_j $ with  the   estimates
$$\|W_j(\Phi_j-id)\|_{\alpha, \Pi\times D_{j+1}},~~\|W_j({\cal D}\Phi_j-Id)W_j^{-1}\|_{\alpha, \Pi\times D_{j+1}} \le
c E_j.$$

Meanwhile, we have a sequence of hamiltonian $H_j=N_j+P_j,$
where $N_j(\xi;I)=\langle \omega_j(\xi),I\rangle$ and $P_j$ satisfies
$$\|P_j\|_{\alpha, \Pi\times D_j}\le \epsilon_j=\alpha_j r_j \rho_j^{\tau'}E_j.$$
For any $j\geq 0,$ $\omega_j\in C^m(\Pi)$ and $\omega_{j+1}=\omega_j+\hat\omega_j$ with $\|\hat\omega_j\|_{\alpha, C^m(\Pi)}\le \epsilon_j/r_j.$

Furthermore, for $\xi\in\Pi_j$ we have
 $$H_{j+1}=H_j\circ\Phi_j
 =N_{j+1}+P_{j+1},$$
Denote by $\Phi^j=\Phi_0\circ\Phi_1\circ\cdots \Phi_{j-1}$ with $\Phi^0=\mbox{id}.$ Then the monotonousness of $\{\Pi_j\}$ shows,  $H_{j}=H\circ\Phi^j $ holds for $\xi\in \Pi_j.$

 \vskip 0.3 true cm
 \noindent
 {\bf Convergence of iteration.} Now we prove the convergence
 of the KAM iteration. In the same way as in \cite{P1,Xu4},
it follows that if $c^\frac12 E_0 \le \frac12,$ then
$$\|W_0 {\cal D} \Phi^j  W_j^{-1} \|_{\alpha, \Pi\times D_j}\le
\prod_{i=1}^{j}(1+c E_j)<2.$$
Therefore,
$$\|W_0(\Phi^{j}-\Phi^{j-1})\|_{\alpha, \Pi_j\times D_j},~\|W_0 {\cal D}(\Phi^{j}-\Phi^{j-1})\|_{\alpha, \Pi\times D_j}\le
c E_j.$$

Let $D_*=D(0,\frac12 s)$ and
$\Phi_*=\lim_{j\to\infty}\Phi^j.$  Since $\Phi^j$ is affine in $I$,
$ \Phi^j$ converges to $\Phi_{*}$ on $D(s/2, r/2)$ with the estimate
 \begin{equation}
 \|W_0(\Phi_*-id)\|_{\alpha, \Pi\times D(s/2,r/2)}\le
 cE_0.\nonumber
\end{equation}

Denote by $P_j\to P_*$ and $\omega_j \to \omega_{*}$.
Then $P_*$  is real analytic with respect to  $(I,\theta)$ on $D(r/2, s/2)$ and $C^m$-smooth in $\xi$ on $\Pi.$ Moreover,
$\frac{\partial^{\ell} P_*}{\partial I^{\ell}}  |_{I=0}=0, \  |\ell|\le 1.$
Note that
$\omega_j=\omega+\sum_{i=0}^{j-1}\hat \omega_i.$ Then    we have
$$\|\omega_{*}-\omega_j\|_{\alpha, C^m(\Pi)}\le  \sum_{i=j}^{\infty}\frac{\epsilon_j}{r_j}=\sum_{i=j}^{\infty}\alpha_i\rho_i^{\tau'}E_i
\le \frac{2\epsilon_j}{r_j}.$$
 Especially,
  $$\|\omega_*-\omega\|_{\alpha, C^m(\Pi)}\le   \frac{2\epsilon}{r}.$$

Let
$\Pi_*=\{\xi\in\Pi: \ \omega_*(\xi)\in O_\alpha \}$. In the sequel we  show $\Pi_{*}\subset \Pi_j$ for all $j\geq 0.$ In fact,
recall $F_j=2^{j+3}\epsilon_jK_j^{\tau+1}/r_j\le \alpha$.
Then, for $\xi\in\Pi_*$ and $0<|k|\le K_j,$
\begin{align*}
|\langle \omega_j, k\rangle|&\ge |\langle \omega_*, k\rangle|-|\langle \omega_*-\omega_j, k\rangle|
  \ge \frac{\alpha}{|k|^{\tau}}-\frac{2\epsilon_j}{r_j}K_j \\
  &\ge \frac{\alpha}{|k|^{\tau}}-\frac{\alpha}{2^{j+2}}\cdot \frac{1}{K_j^{\tau}}\ge \frac{\alpha_j}{|k|^{\tau}}.
\end{align*}
Therefore,  $\Pi_*\subset \bigcap_{j\ge 0}\Pi_j.$
Finally, we arrive at   $H\circ\Phi_*=H_*=N_*+P_* $ for $\xi\in\Pi_*$.




\end{document}